\documentclass{ap-jnmp}
\usepackage{amssymb}
\usepackage{graphicx}
\usepackage{amsmath}
\usepackage{amssymb}
\usepackage[dvips]{color}
\usepackage{booktabs}
\usepackage{bbm}
\usepackage{mathrsfs}

\newtheorem{priteo}{Theorem}[section]
\newtheorem{lema}{Lemma}[section]

\newtheorem{nota}{Remark}[section]

\newtheorem{defi}{Definition}[section]
\newtheorem{prop}{Proposition}[section]

\copyrightauthor{}

\title{Heteroclinic Cycles in Systems with $\mathbb{Z}_2\times\mathbb{Z}_2\times\mathbb{Z}_2$ Symmetry, Revisited}

\author{\footnotesize Adrian C. Murza\footnote{Institute of Mathematics "Simion Stoilow" of the Romanian Academy, P.O. Box $1$--$764,$ RO--$014700,$ Bucharest, Romania}}

\address{Institute of Mathematics "Simion Stoilow" of the Romanian Academy\\
Bucharest, $014700,$ Romania\\
\email{adrian\_murza@hotmail.com}}

\begin{document}

\maketitle

\begin{abstract}
We analyze the generating mechanisms for heteroclinic cycles in $\mathbb{Z}_2\times\mathbb{Z}_2\times\mathbb{Z}_2$--equivariant ODEs, not involving Hopf bifurcations. Such cycles have been observed in particle physics systems with the mentioned symmetry, in absence of the Hopf bifurcation, see \cite{bury} and \cite{Park}, and as far as we know, there is no available theoretical data explaining these phenomena. We use singularity theory to study the equivalence in the group-symmetric context, as well as the recognition problem for the simplest bifurcation problems with this symmetry group. Singularity results highlight different mechanisms for the appearance of heteroclinic cycles, based on the transition between the bifurcating branches. On the other hand, we analyze the heteroclinic cycle of a generic dynamical system with the symmetry of the group $\mathbb{Z}_2\times\mathbb{Z}_2\times\mathbb{Z}_2$ acting on a eight--dimensional torus $\mathbb{T}^8,$ constructed via a Cayley graph, under weak coupling. We identify the conditions for heteroclinic cycle between four equilibria in the three--dimensional fixed point subspaces of some of the isotropy subgroups of $\mathbb{Z}_2\times\mathbb{Z}_2\times\mathbb{Z}_2\times\mathbb{S}^1.$ We also analyze the stability of the heteroclinic cycle.
\end{abstract}

\ccode{2000 Mathematics Subject Classification: 37C80, 37G40, 34C15, 34D06, 34C15}

\section{Introduction}
The interest on the behavior offered by $\mathbb{Z}_2\times\mathbb{Z}_2\times\mathbb{Z}_2$--equivariant ODEs increased considerably over the last decade, due to its possible application in particle physics, more specifically to the neutrino eigenmass determination problem. It has been shown that the eigenmass matrix respects the $\mathbb{Z}_2\times\mathbb{Z}_2\times\mathbb{Z}_2$ symmetry \cite{lashin}-\cite{lashin3}. Moreover, heteroclinic cycles have been described in such systems, see \cite{bury} or \cite{Park}, in absence of Hopf bifurcation. We will show that the bifurcation analysis from the singularity and group theoretical points of view as well as the weak--coupling case can offer valuable insight on the mechanisms leading to heteroclinic cycles in $\mathbb{Z}_2\times\mathbb{Z}_2\times\mathbb{Z}_2$--equivariant systems, in absence of the Hopf bifurcation.

So far, to our best knowledge, the main reference on the heteroclinic cycles appearing in $\mathbb{Z}_2\times\mathbb{Z}_2\times\mathbb{Z}_2$--equivariant systems is Melbourne's work \cite{Melbourne}.
In this article the author analyzes the interaction of three Hopf modes to show that locally a bifurcation
gives rise to heteroclinic cycles between three periodic solutions. More specifically, he considers a vector $f$
field with an equilibrium and assumes that the Jacobian matrix of $f$ about this equilibrium has three distinct complex conjugate pairs of eigenvalues on the imaginary axis. He obtains three branches of periodic solutions arising at the Hopf point from the steady--state equilibrium, as the parameters are varied. He uses Birkhoff normal form, to approximate $f$ close to the bifurcation point by a vector field commuting with the symmetry group of the three-torus.

In this paper we do not assume the existence of three Hopf modes to study the heteroclinic cycles in $\mathbb{Z}_2\times\mathbb{Z}_2\times\mathbb{Z}_2$--equivariant systems. Instead, we perform a detailed analysis of the bifurcation problem with $\mathbb{Z}_2\times\mathbb{Z}_2\times\mathbb{Z}_2$ symmetry, with results from singularity theory. More specifically, after analyze the action of the group $\mathbb{Z}_2\times\mathbb{Z}_2\times\mathbb{Z}_2$ on $\mathbf{R}^3$, we study the restrictions on bifurcation problems $g$ commuting with $\mathbb{Z}_2\times\mathbb{Z}_2\times\mathbb{Z}_2$ symmetry; we analyze the equivalence and
the recognition problem for the simplest bifurcation problems with $\mathbb{Z}_2\times\mathbb{Z}_2\times\mathbb{Z}_2$--symmetry. This allows us to identify two new possible mechanisms for obtaining heteroclinic cycles in these systems. They are based on the smooth transition and jumping between the bifurcating solution branches, respectively. Moreover, we carry out the the linearization of the normal form with $\mathbb{Z}_2\times\mathbb{Z}_2\times\mathbb{Z}_2$--symmetry. This allow us to obtain the explicit form of all possible eigenvalues for this problem. Up to this point, we owe our results to the application of the analysis methods developed in \cite{GS85} and \cite{GS88}.

In the second part of our work, we analyze a third possible mechanism leading to heteroclinic cycles in $\mathbb{Z}_2\times\mathbb{Z}_2\times\mathbb{Z}_2$--equivariant systems: the low coupling case with no Hopf bifurcation. We use the method designing oscillatory networks with the symmetry of a specific group developed by Ashwin and Stork \cite{Stork}, based on the Cayley graph of our group. By considering the weak--coupling case, we reduce the asymptotic dynamics to a flow on an eight--dimensional torus $\mathbb{T}^{8};$ this allows us to average the whole network and introduce an extra $\mathbb{S}^1$ symmetry. As a consequence, we are enabled to identify the existence and classify the stability of heteroclinic cycles in some of the three--dimensional subspaces which are invariant under the group action.

By providing three additional mechanisms capable of generating heteroclinic cycles in $\mathbb{Z}_2\times\mathbb{Z}_2\times\mathbb{Z}_2$--equivariant systems, we believe our paper affords significant additional insight to the original knowledge of these phenomena due to Melbourne \cite{Melbourne}.

This paper is organized as follows. In Section \ref{sec1} we discuss the bifurcation problems with $\mathbb{Z}_2\times\mathbb{Z}_2\times\mathbb{Z}_2$ symmetry from three angles: the group action on $\mathbf{R}^3,$ the restrictions on bifurcation problems $g$ commuting with  this group, and define the solution types. In Section \ref{sec2} we use singularity theory: to study the equivalence in the $\mathbb{Z}_2\times\mathbb{Z}_2\times\mathbb{Z}_2$--symmetric context; to analyze the recognition problem for the simplest bifurcation problems with this symmetry group and to analyze the linearized stability of the normal form. This section involves multiple but straightforward computations, especially in the proof of Theorem \ref{teo42}. We give all the computation outlines, but, because of the dimensions of the involved matrices, we preferred to analyze the outcome rather than filling many pages with unnecessary rows/columns. We conclude this section by identifying two important mechanisms for generating heteroclinic cycles with no need of invoking Hopf bifurcation. Section \ref{sec4} can be viewed as a preparatory section for the weak--coupling case. We show how to construct a generic eight--dimensional $\mathbb{Z}_2\times\mathbb{Z}_2\times\mathbb{Z}_2$--equivariant system, based on the Cayley graph of our group. Section \ref{sec5} is entirely devoted to the weak--coupling (non--Hopf bifurcation involving), based mechanism to produce heteroclinic cycles in $\mathbb{Z}_2\times\mathbb{Z}_2\times\mathbb{Z}_2$--equivariant systems. We show how we can embed the dynamics on an eight--dimensional torus. We assume hyperbolicity of the individual limit cycles for small enough values of the coupling parameter. Moreover, we identify the three--dimensional subspaces invariant under the action of certain isotropy subgroups of $\mathbb{Z}_2\times\mathbb{Z}_2\times\mathbb{Z}_2,$ and prove the existence of heteroclinic cycles within these spaces. We also analyze their stability.

\section{Bifurcation problems with $\mathbb{Z}_2\times\mathbb{Z}_2\times\mathbb{Z}_2$ symmetry}\label{sec1}
In this section we discuss the following points:
\begin{itemize}
\item[(a)] The action of the group $\mathbb{Z}_2\times\mathbb{Z}_2\times\mathbb{Z}_2$ on $\mathbf{R}^3;$
\item[(b)] Restrictions on bifurcation problems $g$ commuting with $\mathbb{Z}_2\times\mathbb{Z}_2\times\mathbb{Z}_2;$
\item[(c)] Solution types of the equation $g=0.$
\end{itemize}
\subsection{Preliminary notations}\label{sectionprelim}
The application of Singularity Theory to our bifurcation analysis requires the use of many concepts developed in \cite{GS85}. To facilitate the reading of this paper, we give here a brief description of them adapted to our case; for more details, the reader is invited to visit the mentioned reference which constitutes the main guidance for this section of the paper.
Let $\mathbf{x}=(x,y,z)\in\mathbf{R}^3.$ For the clarity of the explanations, we will use the explicit expression of the above equation only when it is required by the situation. By
$$\mathscr{E}_{\mathbf{x},\lambda}$$
we denote the space of all functions in three state parameters and one bifurcation parameter $(\lambda),$ that are defined and $C^{\infty}$ on some neighborhood of the origin. A germ is an equivalence class in $\mathscr{E}_{\mathbf{x},\lambda}.$
We denote by
$$\mathscr{E}_{\mathbf{x},\lambda}(\Gamma)$$
the ring of $\Gamma$--equivariant germs. That is, if $f$ is a germ, then
$$f(\gamma\cdot\mathbf{x})=\gamma\cdot f(\mathbf{x}),\forall \mathbf{x}\in\mathbf{R}^3,\forall\gamma\in\Gamma.$$
The module $\overrightarrow{\mathscr{E}}(\Gamma)$ over the ring $\mathscr{E}(\Gamma)$ is defined as in the following Theorem.
\begin{priteo}[Po\'enaru, 1976, \cite{Po}]\label{poenaru}
Let $\Gamma$ be a compact Lie group and let $g_1,\ldots,g_r$ generate the module $\overrightarrow{\mathscr{P}}(\Gamma)$ of $\Gamma$--equivariant polynomials over the ring $\mathscr{P}(\Gamma).$ Then $g_1,\ldots,g_r$ generate the module $\overrightarrow{\mathscr{E}}(\Gamma)$ over the ring $\mathscr{E}(\Gamma).$
\end{priteo}

Moreover, as in \cite{GS85}, we have the following definition for $\overleftrightarrow{\mathscr{E}}_{\mathbf{x},\lambda}(\Gamma)$:
$$\overleftrightarrow{\mathscr{E}}_{\mathbf{x},\lambda}(\Gamma)=\left\{3\times 3~\mathrm{matrix~germs}~S(\mathbf{x},\lambda):S(\gamma\cdot\mathbf{x},\lambda)=\gamma\cdot S(\mathbf{x},\lambda)\right\}.$$
We define
$$\overrightarrow{\mathscr{M}}_{\mathbf{x},\lambda}(\Gamma)=\left\{g\in\overrightarrow{\mathscr{E}}_{\mathbf{x},\lambda}(\Gamma):
g(\mathbf{0},0)=0\right\};$$
that is, $\overrightarrow{\mathscr{M}}_{\mathbf{x},\lambda}(\Gamma)$ consists of $\Gamma$--equivariant mappings that vanish at the origin.
Finally, we need to define the $\Gamma$--equivariant restricted tangent space $RT(h,\Gamma)$ of a $\Gamma$--equivariant bifurcation problem $h\in\overrightarrow{\mathscr{E}}(\Gamma).$ In order to do this, we have to give first the following definition.
\begin{defi}\label{dfinsing1}
Let $g,h:\mathbf{R}^3\times\mathbf{R}\rightarrow\mathbf{R}^3,~g,h\in\overrightarrow{\mathscr{E}}_{\mathbf{x},\lambda}(\Gamma)$ be a bifurcation problem with three state variables. Then $g$ and $h$ are equivalent if there exists an invertible change of coordinates $(\mathbf{x},\lambda)\mapsto(Z(\mathbf{x},\lambda),\Lambda(\lambda))$  and $S$ is a $3\times 3$ invertible matrix depending smoothly on $\mathbf{x},$ such that
\begin{equation}\label{eq1def1}
g(\mathbf{x},\lambda)=S(\mathbf{x},\lambda)h(Z(\mathbf{x},\lambda),\Lambda(\lambda))
\end{equation}
where the mapping $\Phi(\mathbf{x},\lambda)=(Z(\mathbf{x},\lambda),\Lambda(\lambda))$ is preserving the orientation in $\lambda;$ in particular,
\begin{equation}\label{eq2def1}
\begin{array}{l}
\hspace{1.5cm}Z(\mathbf{0})=0,~~\Lambda(0)=0,~~\mathrm{det}(dZ)_{(\mathbf{0})}\neq0,~~\Lambda'(0)>0~~\mathrm{and}\\
\\
Z(\gamma\cdot\mathbf{x},\lambda)=\gamma\cdot Z(\mathbf{x},\lambda),~~
S(\gamma\cdot\mathbf{x},\lambda)=\gamma\cdot S(\mathbf{x},\lambda),~~
S(\mathbf{0},0),(dZ)_{\mathbf{0},0}\in\mathscr{L}_{\Gamma}(V)^0.
\end{array}
\end{equation}
We call $g$ and $h$ strongly $\Gamma$--equivalent if $\Lambda(\lambda)\equiv\lambda.$
\end{defi}
We define the $\Gamma$--equivariant restricted tangent space of $g$ to be
$$RT(g,\Gamma)=\left\{S\cdot g+(dg)Z,~Z(\mathbf{0},0)=0\right\},$$
while $S(\mathbf{x},\lambda)$ and $Z$ satisfy \eqref{eq2def1}.
In all of these cases, of course $\Gamma=\mathbb{Z}_2\times\mathbb{Z}_2\times\mathbb{Z}_2.$
\begin{lemma}[Nakayama's Lemma, \cite{GS85}]\label{lemanaka}
Let $\mathscr{I}$ and $\mathscr{J}$ be ideals in $\mathscr{E}_n,$ and assume that $\mathscr{I}=\langle p_1,\ldots,p_l\rangle$ is finitely generated. Then $\mathscr{I}\subset\mathscr{J}$ if and only if $\mathscr{I}\subset\mathscr{J}+\mathscr{M}\cdot\subset\mathscr{I}.$
\end{lemma}
\subsection{The action of $\mathbb{Z}_2\times\mathbb{Z}_2\times\mathbb{Z}_2$ on $\mathbf{R}^3$}
The group $\mathbb{Z}_2\times\mathbb{Z}_2\times\mathbb{Z}_2$ has eight elements $(\kappa,\zeta,\xi)$ where $\kappa=\pm1,$ $\zeta=\pm1,$ $\xi=\pm1.$ The group element $(\kappa,\zeta,\xi)$ acts on the point
$(x,y,z)\in\mathbf{R}^3$ by
$$(\kappa,\zeta,\xi)\cdot(x,y,z)=(\kappa x,\zeta y,\xi z).$$
We may think of the action of $(\kappa,\zeta,\xi)$ on $\mathbf{R}^3$ as a linear mapping; the matrix associated to the action $(\kappa,\zeta,\xi)$ is the diagonal matrix
\begin{equation}
\begin{array}{l}
\begin{bmatrix}
\kappa&0&0\\
0&\zeta&0\\
0&0&\xi
\end{bmatrix}.
\end{array}
\end{equation}
The behavior of the action of $\mathbb{Z}_2\times\mathbb{Z}_2\times\mathbb{Z}_2$ on $\mathbf{R}^3$ is different at different points in $\mathbf{R}^3.$ We describe these differences in two ways: through orbits and through isotropy subgroups.\\
The orbit of a point $(x,y,z)$ under the action of $\mathbb{Z}_2\times\mathbb{Z}_2\times\mathbb{Z}_2$ is the set of points
$$\left\{(\kappa,\zeta,\xi)\cdot(x,y,z):(\kappa,\zeta,\xi)\in\mathbb{Z}_2\times\mathbb{Z}_2\times\mathbb{Z}_2\right\}.$$
There are eight orbit types:
\begin{equation}\label{orbits}
\begin{array}{l}
\hspace{-7.4cm}(a)~\mathrm{The}~\mathrm{origin},~(0,0,0),\\
\hspace{-7.4cm}(b)~\mathrm{Points}~\mathrm{on}~\mathrm{the}~x-\mathrm{axis},~~(\pm x,0,0)~\mathrm{with}~x\neq0,\\
\hspace{-7.4cm}(c)~\mathrm{Points}~\mathrm{on}~\mathrm{the}~y-\mathrm{axis},~~( 0,\pm y,0)~\mathrm{with}~y\neq0,\\
\hspace{-7.4cm}(d)~\mathrm{Points}~\mathrm{on}~\mathrm{the}~z-\mathrm{axis},~~(0,0,\pm z)~\mathrm{with}~z\neq0,\\
\hspace{-7.4cm}(e)~\mathrm{Points}~\mathrm{on}~\mathrm{the}~\mathrm{plane}~x=0,~~(0,\pm y,\pm z)~\mathrm{with}~y\neq0,~z\neq0\\
\hspace{-7.4cm}(f)~\mathrm{Points}~\mathrm{on}~\mathrm{the}~\mathrm{plane}~y=0,~~(\pm x,0,\pm z)~\mathrm{with}~x\neq0,~z\neq0\\
\hspace{-7.4cm}(g)~\mathrm{Points}~\mathrm{on}~\mathrm{the}~\mathrm{plane}~z=0,~~(\pm x,\pm y,0)~\mathrm{with}~x\neq0,~y\neq0\\
\hspace{-7.4cm}(h)~\mathrm{Points}~\mathrm{off}~\mathrm{the}~\mathrm{axes},~(\pm x,\pm y,\pm z)~\mathrm{with}~x\neq0,~y\neq0,~z\neq0.
\end{array}
\end{equation}
We see that orbits have either $1,~2,~4$ or $8$ points, the origin being the unique one-orbit point.
The isotropy subgroup of a point $(x,y,z)$ is the set of symmetries preserving that point. In symbols, the isotropy subgroup of the point $(x,y,z)$ is
$$\left\{(\kappa,\zeta,\xi)\in\mathbb{Z}_2\times\mathbb{Z}_2\times\mathbb{Z}_2:(\kappa,\zeta,\xi)\cdot(x,y,z)=(x,y,z)\right\}.$$
It is easy to see that there are eight isotropy subgroups:
\begin{itemize}
\vspace{0.25cm}
\item[(a)] $\mathbb{Z}_2\times\mathbb{Z}_2\times\mathbb{Z}_2~\mathrm{corresponding}~\mathrm{to}~\mathrm{the}~\mathrm{origin},$
\item[(b)] $\mathbb{Z}_2\times\mathbb{Z}_2=\left\{1,\zeta,\xi\right\}~\mathrm{corresponding}~\mathrm{to}~(x,0,0)~\mathrm{with}~x\neq0,$
\item[(c)] $\mathbb{Z}_2\times\mathbb{Z}_2=\left\{\kappa,1,\xi\right\}~\mathrm{corresponding}~\mathrm{to}~(0,y,0)~\mathrm{with}~y\neq0,$
\item[(d)] $\mathbb{Z}_2\times\mathbb{Z}_2=\left\{\kappa,\zeta,1\right\}~\mathrm{corresponding}~\mathrm{to}~(0,0,z)~\mathrm{with}~x\neq0,$
\item[(e)] $\mathbb{Z}_2=\left\{\kappa,1,1\right\}~\mathrm{corresponding}~\mathrm{to}~(0,y,z)~\mathrm{with}~y\neq0,~z\neq0,$
\item[(f)] $\mathbb{Z}_2=\left\{1,\zeta,1\right\}~\mathrm{corresponding}~\mathrm{to}~(x,0,z)~\mathrm{with}~x\neq0,~z\neq0,$
\item[(g)] $\mathbb{Z}_2=\left\{1,1,\xi\right\}~\mathrm{corresponding}~\mathrm{to}~(x,y,0)~\mathrm{with}~x\neq0,~y\neq0,$
\item[(h)] $\mathbbm{1}=\left\{1,1,1\right\}~\mathrm{corresponding}~\mathrm{to}~(x,y,z)~\mathrm{with}~x\neq0,~y\neq0,~z\neq0.$
\end{itemize}

\subsection{The form of $\mathbb{Z}_2\times\mathbb{Z}_2\times\mathbb{Z}_2$--symmetric bifurcation problems}
Let $g:\mathbf{R}^3\times\mathbf{R}\rightarrow\mathbf{R}^3$ be a bifurcation problem with three state variables; that is, let $g$ be $C^{\infty}$ and satisfy
$$g(0,0,0,0)=0,~~(dg)_{(0,0,0,0)}=0.$$
We say that the bifurcation problem $g$ commutes with $\mathbb{Z}_2\times\mathbb{Z}_2\times\mathbb{Z}_2$ if
\begin{equation}\label{commut}
g((\kappa,\zeta,\xi)\cdot(x,y,z),\lambda)=(\kappa,\zeta,\xi)\cdot g(x,y,z,\lambda).
\end{equation}
We will need the following result for the next lemma.
\begin{lema}\label{lema normal form1}
If $f\in\mathscr{E}_{\mathbf{x},\lambda}$ is even in $\mathbf{x},$ then $f$ may be expressed as a smooth function of $\mathbf{x}^2$ and $\lambda;$ in symbols,
$$f(\mathbf{x},\lambda)=a(\mathbf{x}^2,\lambda).$$
\end{lema}
\begin{proof}
See Lema VI, 2.1 page 248 in \cite{GS85}.
\end{proof}
We can now state
\begin{lema}\label{lema normal form}
Let us consider the $g:\mathbf{R}^3\times\mathbf{R}\rightarrow\mathbf{R}$ bifurcation problem in three state variables commuting with the action of $\mathbb{Z}_2\times\mathbb{Z}_2\times\mathbb{Z}_2$. Then there exist smooth functions $p(u,v,w,\lambda),~q(u,v,w,\lambda),~r(u,v,w,\lambda)$ such that\\
\begin{equation}\label{eq lema normal form}
\begin{array}{l}
g(x,y,z,\lambda)=(p(x^2,y^2,z^2,\lambda)x,~q(x^2,y^2,z^2,\lambda)y,~r(x^2,y^2,z^2,\lambda)z)~\mathrm{where}\\
\\
\hspace{1.8cm}p(0,0,0,0)=0,~~q(0,0,0,0)=0,~~r(0,0,0,0)=0.
\end{array}
\end{equation}
\end{lema}

\begin{proof}

We write in $g$ coordinates
\begin{equation}\label{eq lema g}
g(x,y,z,\lambda)=(a(x,y,z,\lambda),~b(x,y,z,\lambda),~c(x,y,z,\lambda)).
\end{equation}
Commutativity with equation \eqref{commut} implies
\begin{equation}\label{eq lema normal form2}
\begin{array}{l}
a(\kappa x,\zeta y,\xi z,\lambda)=\kappa a(x,y,z,\lambda),~~
b(\kappa x,\zeta y,\xi z, \lambda)=\zeta b(x,y,z,\lambda),~~
c(\kappa x,\zeta y,\xi z, \lambda)=\xi c(x,y,z,\lambda).
\end{array}
\end{equation}
The action of $\xi$ is defined by $(x,y,z)\rightarrow(x,y,-z).$ Now $\kappa$ transforms $z$ into $\bar{z}$, i.e. $(x,y,z)\rightarrow(x,-y,z)$ and the action of $\zeta$ is $+1.$ When $\kappa=-1,~\zeta=+1,~\xi=+1,$ equation \eqref{eq lema normal form2} shows that $a$ is odd in $x$ while $b$ and $c$ are even in $x,$ respectively and $c$ is even in $z$. When $\kappa=+1,~\zeta=-1,~\xi=+1,$ equation \eqref{eq lema normal form2} shows that $a$ is even in $y,$ $b$ is odd in $y$,  while $c$ is even in $y$ and $z.$\\
Conversely, if $\kappa=-1,~\zeta=+1,~\xi=-1,$ we get that $a$ is odd in $x$ while $b$ is even in $x$ and $c$ is odd in $x$ and $z$ while  when $\kappa=1,~\zeta=+1,~\xi=-1$ we get that $a$ is even in $y,$ $b$ is odd in $y$,  while $c$ is odd in $y$ and $z.$ It follows from the Taylor's theorem that we may factor these functions

\begin{equation}\label{factor}
\begin{array}{l}
a(x,y,z,\lambda)=\bar{a}(x,y,z,\lambda)x,~~~b(x,y,z,\lambda)=\bar{b}(x,y,z,\lambda)y,~~~
c(x,y,z,\lambda)=\bar{c}(x,y,z,\lambda)z\\
\end{array}
\end{equation}
where $\bar{a},~\bar{b}$ and $\bar{c}$ are even in $x,~y$ and $z.$ Applying Lemma \ref{lema normal form1} first to $x$, then to $y$ and finally to $z$ we conclude that $g$ has the desired form \eqref{eq lema normal form}.
The linear terms in $g$ vanish. The only linear terms compatible with the symmetry are
$$(p(0,0,0,0)x,~q(0,0,0,0)y,~r(0,0,0,0)z);$$
thus, $p(0,0,0,0)=q(0,0,0,0)=r(0,0,0,0)=0.$
\end{proof}

\subsection{Solution types for $g$}
Consider solving the equation $g=0$ when $g$ has the form \eqref{eq lema normal form}. There are eight solution types which occur according as the first, the second or the third factor in $p(x^2,y^2,z^2,\lambda)x$ vanishes, the first, the second or the third factor in $q(x^2,y^2,z^2,\lambda)y$ vanishes or the first, the second or the third factor in $r(x^2,y^2,z^2,\lambda)r$ vanishes. Specifically, we have the solution types
\begin{itemize}
\vspace{0.25cm}
\item[(a)] x=y=z=0,
\item[(b)] $p(x^2,0,0,\lambda)=0,$ $y=z=0,~x\neq0$,
\item[(c)] $q(0,y^2,0,\lambda)=0,$ $x=z=0,~y\neq0$,
\item[(d)] $r(0,0,z^2,\lambda)=0,$ $x=y=0,~z\neq0$,
\item[(e)] $p(x^2,y^2,0,\lambda)=0,$ $q(x^2,y^2,0,\lambda)=0,$ $z=0,~x\neq0~y\neq0$,
\item[(f)] $p(x^2,0,z^2,\lambda)=0,$ $r(x^2,0,z^2,\lambda)=0,$ $y=0,~x\neq0~z\neq0$,
\item[(g)] $q(0,y^2,z^2,\lambda)=0,$ $r(0,y^2,z^2,\lambda)=0,$ $x=0,~y\neq0~z\neq0$,
\item[(h)] $p(x^2,y^2,z^2,\lambda)=0,~q(x^2,y^2,z^2,\lambda)=0,~r(x^2,y^2,z^2,\lambda)=0,$ $x\neq0,~y\neq0,~z\neq0.$
\end{itemize}
These solution types correspond exactly to the orbit types listed in \eqref{orbits} of the action of $\mathbb{Z}_2\times\mathbb{Z}_2\times\mathbb{Z}_2$ on $\mathbf{R}^3.$ As in \cite{GS85} we use the following terminology for these five types of solutions:
\begin{itemize}
\vspace{0.25cm}
\item[(a)] trivial solutions,
\item[(b)] $x$--mode solutions,
\item[(c)] $y$--mode solutions,
\item[(d)] $z$--mode solutions,
\item[(e)] $xy$--mixed mode solutions,
\item[(f)] $xz$--mixed mode solutions,
\item[(g)] $yz$--mixed mode solutions,
\item[(h)] $xyz$--mixed mode solutions.
\end{itemize}
Each solution type has its own characteristic multiplicity. The $x$--mode, $y$--mode and $z$--mode solutions always come in pairs $(\pm x,0,0)$, $(0,\pm y,0)$, $(0,0,\pm z)$ and mixed mode solutions on the one hand four at the time, and they are $(\pm x,\pm y,0),~(\pm x,0,\pm z)$ and $(0,\pm y,\pm z),$ while $(\pm x,\pm y,\pm z)$ come eight at a time.

\section{Singularity results}\label{sec2}
We divide this section into three subsections:
\begin{itemize}
\vspace{0.25cm}
\item[(1)] Equivalence in the $\mathbb{Z}_2\times\mathbb{Z}_2\times\mathbb{Z}_2$-- symmetric context;
\item[(2)] The recognition problem for the simplest bifurcation problems with $\mathbb{Z}_2\times\mathbb{Z}_2\times\mathbb{Z}_2$-- symmetry;
\item[(3)] Linearized stability and $\mathbb{Z}_2\times\mathbb{Z}_2\times\mathbb{Z}_2$ symmetry.
\end{itemize}
\subsection{$\mathbb{Z}_2\times\mathbb{Z}_2\times\mathbb{Z}_2$--equivalence}\label{subsecequiv}
The singularities we describe here have codimension eight and modality six. We have the following remarks regarding Definition \ref{dfinsing1}.
\begin{nota}
Since S in \eqref{dfinsing1} is invertible, we see that
\begin{equation}\label{dfinsing2}
\begin{array}{l}
\Phi(\left\{(z,\lambda):g(\mathbf{x},\lambda)=0\right\})=\left\{(\mathbf{x},\lambda):h(\mathbf{x},\lambda)=0\right\}.
\end{array}
\end{equation}
Thus equivalences preserve bifurcation diagrams. They also preserve the orientation of the parameter $\lambda.$
\end{nota}
Let $g,h:\mathbf{R}^3\times\mathbf{R}\rightarrow\mathbf{R}^3$ be a bifurcation problem with three state variables commuting with the action of $\mathbb{Z}_2\times\mathbb{Z}_2\times\mathbb{Z}_2.$ We say that $g$ and $h$ are $\mathbb{Z}_2\times\mathbb{Z}_2\times\mathbb{Z}_2$--equivalent if $g$ and $h$ are equivalent in the sense of the Definition \ref{dfinsing1}, and in addition the equivalence preserves the symmetry. Recall that $g$ and $h$ are equivalent if there exists a $3\times3$ invertible matrix $S(x,y,z,\lambda)$ depending smoothly on $x,$ $y,$ $z$ and $\lambda$ and a diffeomorphism $\Phi(x,y,z,\lambda)=(Z(x,y,z,\lambda),\Lambda(\lambda))$ satisfying
\begin{equation}\label{diffeo}
\begin{array}{l}
g(x,y,z,\lambda)=S(x,y,z,\lambda)h(Z(x,y,z,\lambda),\Lambda(\lambda))\\
\end{array}
\end{equation}
such that
\begin{equation}\label{diffeo1}
\Phi(0,0,0,0)=(0,0,0,0)~\mathrm{and}~\Lambda'(0)>0.
\end{equation}
We say that the equivalence $S,~\Phi$ preserves the symmetry if
\begin{equation}\label{preserv symm}
\begin{array}{l}
(a) ~~Z(\kappa x,\zeta y,\xi z,\lambda)=(\kappa,\zeta,\xi)\cdot Z(x,y,z,\lambda),\\
\\
(b)~~ S(\kappa x,\zeta y,\xi z,\lambda)\begin{bmatrix}\kappa&0&0\\0&\zeta&0\\0&0&\xi\end{bmatrix}=
\begin{bmatrix}\kappa&0&0\\0&\zeta&0\\0&0&\xi\end{bmatrix}S(x,y,z,\lambda).
\end{array}
\end{equation}
Condition \eqref{preserv symm} restricts the form of $Z$ and $S$ in the following ways. Applying Lemma \ref{lema normal form} to show that
\begin{equation}
Z(x,y,z,\lambda)=(a(x^2,y^2,z^2,\lambda)x,b(x^2,y^2,z^2,\lambda)y,c(x^2,y^2,z^2,\lambda)z).
\end{equation}
Therefore
\begin{equation}
\begin{array}{l}
(dZ)_{(0,0,0,0)}=\begin{bmatrix}
a(0,0,0,0) &0&0\\
0&b(0,0,0,0)&0\\
0&0&c(0,0,0,0)
\end{bmatrix}
\end{array};
\end{equation}
i.e. $(dZ)_{(0,0,0,0)}$ is diagonal. Dealing now with $S$, we write out entries of $S$ as
\begin{equation}\label{matrixS}
\begin{array}{l}
\begin{bmatrix}
S_1(x,y,z,\lambda)&S_2(x,y,z,\lambda)&S_3(x,y,z,\lambda)\\
S_4(x,y,z,\lambda)&S_5(x,y,z,\lambda)&S_6(x,y,z,\lambda)\\
S_7(x,y,z,\lambda)&S_8(x,y,z,\lambda)&S_9(x,y,z,\lambda)
\end{bmatrix}.
\end{array}
\end{equation}
A calculation using \eqref{preserv symm} $(b)$ shows that $S_1,~S_5$ and $S_9$ are even in $x,~y$ and $z$, while $S_2,~S_3,~S_4,~S_6,~S_7$ and $S_8$ are odd in $x,~y$ and $z$. Therefore, Lemma \ref{lema normal form1} together with Taylor's theorem implies that
\begin{equation}\label{matrixS1}
S(x,y,z,\lambda)=
\begin{array}{l}
\begin{bmatrix}
d_1(x^2,y^2,z^2,\lambda)&d_2(x^2,y^2,z^2,\lambda)xyz&d_3(x^2,y^2,z^2,\lambda)xyz\\
d_4(x^2,y^2,z^2,\lambda)xyz&d_5(x^2,y^2,z^2,\lambda)&d_6(x^2,y^2,z^2,\lambda)xyz\\
d_7(x^2,y^2,z^2,\lambda)xyz&d_8(x^2,y^2,z^2,\lambda)xyz&d_9(x^2,y^2,z^2,\lambda)
\end{bmatrix}.
\end{array}
\end{equation}
In particular $S(0,0,0,0)$ is diagonal and has the form
\begin{equation}\label{matrixS1ty}
S(x,y,z,\lambda)=
\begin{array}{l}
\begin{bmatrix}
d_1(0,0,0,0)&0&0\\
0&d_5(0,0,0,0)&0\\
0&0&d_9(0,0,0,0)
\end{bmatrix}.
\end{array}
\end{equation}
In order to have $\mathbb{Z}_2\times\mathbb{Z}_2\times\mathbb{Z}_2$--equivalences preserved linear stability (wihich will be discussed in detail in the next section), we shall require that $\mathbb{Z}_2\times\mathbb{Z}_2\times\mathbb{Z}_2$--equivalences satisfy
\begin{equation}\label{shall}
\begin{array}{l}
a(0,0,0,0)>0,~~b(0,0,0,0)>0,~~c(0,0,0,0)>0,~~\\
\\
d_1(0,0,0,0)>0,~~d_5(0,0,0,0)>0,~~d_9(0,0,0,0)>0.
\end{array}
\end{equation}
So far we have proved
\begin{prop}
Two bifurcation problems $g$ and $h,$ both commuting with the group $\mathbb{Z}_2\times\mathbb{Z}_2\times\mathbb{Z}_2,$ are $\mathbb{Z}_2\times\mathbb{Z}_2\times\mathbb{Z}_2$--equivalent if there exists $S$ and $\Phi=(Z,\Lambda)$ as above satisfying \eqref{diffeo}, \eqref{diffeo1}, \eqref{preserv symm} \eqref{matrixS}, \eqref{matrixS1} and \eqref{matrixS1ty}.
\end{prop}
\subsection{The recognition problem for the simplest examples}
Let $g$ be a bifurcation problem with three state variables commuting with the group $\mathbb{Z}_2\times\mathbb{Z}_2\times\mathbb{Z}_2.$ Thus $g$ has the form \eqref{eq lema normal form}. We split off the lowest terms in \eqref{eq lema normal form}, i.e.
\begin{equation}\label{lowesta}
g(x,y,z,\lambda)=k(x,y,z,\lambda)+\mathrm{hot}
\end{equation}
where
\begin{equation}\label{lowestb}
k(x,y,z,\lambda)=(Ax^3+Bxy^2+Cxz^2+\alpha\lambda x,Dyx^2+Ey^3+Fyz^2+\beta\lambda y,Gzx^2+Hzy^2+Iz^3+\gamma\lambda z).
\end{equation}
The higher-order terms in \eqref{lowesta} include monomials $x^ry^s\lambda^t,~x^rz^s\lambda^t$ and $y^rz^s\lambda^t$ satisfying at least one of the following conditions:
\begin{itemize}
\vspace{0.25cm}
\item[(a)] $r+s\geqslant5,$
\item[(b)] $t=1, ~r+s\geqslant3,$
\item[(c)] $t\geqslant2.$
\end{itemize}
Before proceeding with our analysis, we shall introduce the notion of nondegenerate bifurcation problem in three state variables. The bifurcation is nondegenerate if it satisfies several inequalities which are invariants of equivalence. Our first nondegeneracy condition is
\begin{equation}\label{H1}
g_{\lambda}(0,0,0)\neq0.
\end{equation}
Let
$$\bar{k}(x,y,z)=(Ax^2+By^2+Cz^2,Dx^2+Ey^2+Fz^2,Gx^2+Hy^2+Iz^2).$$
Then our second nondegeneracy condition is
\begin{equation}\label{H2}
\mathrm{minors}(\mathrm{det}(J(\bar{k})))\neq0,
\end{equation}
where $J(\bar{k})$ is the Jacobian matrix of $\bar{k}.$ This condition is to constrain the three roots of the determinant of the Jacobian matrix to be different. Taking into account the generic nondegeneracy conditions \eqref{H2} and (as it will be seen in Theorem \ref{teoprinc}), the specific conditions dictated by the choice of the parameters in \eqref{propnon1}, we have the following definition.
\begin{defi}
The bifurcation problem $g$ in \eqref{lowesta}--\eqref{lowestb} is nondegenerate if all the following conditions are satisfied:
\begin{equation}\label{conditions}
\begin{array}{l}
\hspace{3cm}A\neq0,~E\neq0,~I\neq0,~\alpha\neq0,~\beta\neq0,~\gamma\neq0,\\
\\
\hspace{0.25cm}B|\beta|\neq|E\alpha|,~D|\alpha|\neq|A\beta|,~G|\alpha|\neq|A\gamma|,
C|\gamma|\neq|I\alpha|,~F|\gamma|\neq|I\beta|,~H|\beta|\neq|E\gamma|,\\
\\
\hspace{3.6cm}AE\neq BD,~BF\neq CE,~AF\neq CD.
\end{array}
\end{equation}
\end{defi}
Our main goal in this subsection is to state and prove the Theorem \ref{teoprinc}, in which we solve the recognition problem for nondegenerate bifurcation problems commuting with the $\mathbb{Z}_2\times\mathbb{Z}_2\times\mathbb{Z}_2.$ However, before even stating it, we need a sequence of three preliminary results. The first couple of results recall the Theorem XIV 1.3 and the Proposition XIV 1.4, both from \cite{GS85}, whose proofs can be found in the same reference. The third result is one of our particular developments; while it stands as a result on its own, it also constitutes the second part of the proof of Theorem \ref{teoprinc}. We have:
\begin{priteo}[Theorem XIV 1.3, \cite{GS85}]\label{teo13}
Let $\Gamma$ be a compact Lie group acting on $V.$ Let $h\in\overrightarrow{\mathscr{E}}_{\mathbf{x},\lambda}(\Gamma)$ be a $\Gamma$--equivariant bifurcation problem and let $p$ be any germ in $\overrightarrow{\mathscr{E}}_{\mathbf{x},\lambda}(\Gamma).$ Suppose that
$$RT(h+tp,\Gamma)=RT(h,\Gamma)$$
for all $t\in[0,1].$ Then $h+tp$ is strongly $\Gamma$--equivalent to $h$ for all $t\in[0,1].$
\end{priteo}

\begin{prop}[Proposition  XIV 1.4, \cite{GS85}]\label{prop14}
Let $\Gamma$ be a compact Lie group acting on $V$ and let $h\in\overrightarrow{\mathscr{E}}_{\mathbf{x},\lambda}(\Gamma).$ Then $RT(h,\Gamma)$
is a finitely generated submodule of $\overrightarrow{\mathscr{E}}_{\mathbf{x},\lambda}(\Gamma)$ over the ring $\mathscr{E}_{\mathbf{x},\lambda}(\Gamma).$ Moreover, $RT(h,\Gamma)$ is generated by
$$S_1h,\ldots,S_th;(dh)(X_1),\ldots,(dh)(X_s)$$
where $S_1,\ldots,S_t$ generate $\overleftrightarrow{\mathscr{E}}_{x,\lambda}(\Gamma)$ and $X_1,\ldots,X_s$ generate $\overrightarrow{\mathscr{M}}_{\mathbf{x},\lambda}(\Gamma).$
\end{prop}

\begin{priteo}\label{teo42}
Suppose that $g$ is a nondegenerate $\mathbb{Z}_2\times\mathbb{Z}_2\times\mathbb{Z}_2$--equivariant bifurcation problem. Then $g$ is strongly $\mathbb{Z}_2\times\mathbb{Z}_2\times\mathbb{Z}_2$--equivalent to $h.$
\end{priteo}
\begin{proof}
To perform the proof, we choose to work with $\mathbb{Z}_2\times\mathbb{Z}_2\times\mathbb{Z}_2$--invariant coordinates. For this purpose, we need to find the generators for $RT(h,\mathbb{Z}_2\times\mathbb{Z}_2\times\mathbb{Z}_2).$ In Lemma \ref{lema normal form} we have taken care of the generators for $\mathscr{E}(\mathbb{Z}_2\times\mathbb{Z}_2\times\mathbb{Z}_2)$ and $\overrightarrow{\mathscr{E}}(\mathbb{Z}_2\times\mathbb{Z}_2\times\mathbb{Z}_2).$ Once the generators for $RT(h,\mathbb{Z}_2\times\mathbb{Z}_2\times\mathbb{Z}_2)$ are computed, then by working in invariant coordinates, the action of $\mathbb{Z}_2\times\mathbb{Z}_2\times\mathbb{Z}_2$ is effectively annihilated.
Our purpose is to show that under the assumption of nondegeneracy,
\begin{equation}\label{RT}
RT(h+t\varphi,\mathbb{Z}_2\times\mathbb{Z}_2\times\mathbb{Z}_2)=RT(h,\mathbb{Z}_2\times\mathbb{Z}_2\times\mathbb{Z}_2),~\forall t\in\mathbf{R}.
\end{equation}
Then we will apply Theorem \ref{teo13} to complete the proof.
We start by identifying (working with invariant coordinates), an isomorphism between $\overrightarrow{\mathscr{E}}_{x,y,z,\lambda}(\mathbb{Z}_2\times\mathbb{Z}_2\times\mathbb{Z}_2)$ and $\overrightarrow{\mathscr{E}}_{u,v,w,\lambda}(\mathbb{Z}_2\times\mathbb{Z}_2\times\mathbb{Z}_2),$ where $u=x^2,~v=y^2$ and $w=z^2.$ That is,
$$g(x,y,z,\lambda)=(p(x^2,y^2,z^2,\lambda)x,q(x^2,y^2,z^2,\lambda)y,r(x^2,y^2,z^2,\lambda)z).$$
We write $g$ in the form $[p(u,v,w,\lambda),q(u,v,w,\lambda),r(u,v,w,\lambda)]$ and work in $\overrightarrow{\mathscr{E}}_{u,v,w,\lambda}(\mathbb{Z}_2\times\mathbb{Z}_2\times\mathbb{Z}_2)$ which is a module over $\mathscr{E}_{u,v,w,\lambda}(\mathbb{Z}_2\times\mathbb{Z}_2\times\mathbb{Z}_2).$
A short calculation shows that the nine generators of $\overleftrightarrow{\mathscr{E}}(\mathbb{Z}_2\times\mathbb{Z}_2\times\mathbb{Z}_2)$ are the $3\times3$ matrices $S_k,k=1,\ldots,9,$ each with eight zero entries while the ninth entry is $s_{ij}=1~\mathrm{if}~i=j,$ $s_{ij}=s_{ji}~\mathrm{if}~i\neq j$ and $s_{12}=xy,$ $s_{13}=xz,$ $s_{23}=yz.$ Moreover, one observes that $RT(g,\mathbb{Z}_2\times\mathbb{Z}_2\times\mathbb{Z}_2)$ can be viewed as a submodule of $\overrightarrow{\mathscr{E}}_{u,v,w,\lambda}(\mathbb{Z}_2\times\mathbb{Z}_2\times\mathbb{Z}_2),$ which has the following twelve generators:
\begin{equation}\label{RT1}
\begin{array}{l}
\hspace{0.35cm}[p,0,0],~[0,q,0],~[0,0,r],~[qv,0,0],~[rw,0,0],~[0,pu,0],~[0,rw,0],\\
\\
~[0,0,pu],~[0,0,qv],~[up_u,uq_u,ur_u],~[vp_v,vq_v,vr_v],~[wp_w,wq_w,wr_w].
\end{array}
\end{equation}
We need to show that
\begin{equation}\label{RT2}
\mathscr{M}^2_{u,v,w,\lambda}\overrightarrow{\mathscr{E}}_{u,v,w,\lambda}\subset RT(h+t\varphi,\mathbb{Z}_2\times\mathbb{Z}_2\times\mathbb{Z}_2).
\end{equation}
In order to do this, let $\mathscr{I}\subset RT(g,\mathbb{Z}_2\times\mathbb{Z}_2\times\mathbb{Z}_2)$ be the submodule with the twenty--seven generators
\begin{equation}\label{RT3}
\begin{array}{l}
\hspace{0.1cm}\nu[p,0,0],~\nu[0,q,0],~\nu[0,0,r],~\nu[qv,0,0],~\nu[rw,0,0],~\nu[0,pu,0],~\nu[0,rw,0],\\
\\
~\nu[0,0,pu],~\nu[0,0,qv],~\nu[up_u,uq_u,ur_u],~\nu[vp_v,vq_v,vr_v],~\nu[wp_w,wq_w,wr_w]
\end{array}
\end{equation}
where $\nu=u,v,w$ or $\lambda$ and $g=h+t\varphi.$ We claim that
\begin{equation}\label{RT4}
\mathscr{M}^2_{u,v,w,\lambda}\overrightarrow{\mathscr{E}}_{u,v,w,\lambda}=\mathscr{I}.
\end{equation}
If \eqref{RT4} is true then \eqref{RT2} is also true. In addition, if \eqref{RT4} is true then
\begin{equation}\label{RT5}
(a)\hspace{0.4cm}RT(h+t\varphi,\mathbb{Z}_2\times\mathbb{Z}_2\times\mathbb{Z}_2)=\mathscr{M}^2_{u,v,w,\lambda}\overrightarrow{\mathscr{E}}_{u,v,w,\lambda}+W,
\end{equation}
where
\begin{equation}\label{RT6}
(b)\hspace{0.4cm}W=\mathbf{R}\{[p,0,0],[0,q,0],[0,0,r],u[p_u,q_u,r_u],v[p_v,q_v,r_v],w[p_w,q_w,r_w]\}.
\end{equation}
We compute now the elements composing the basis of $W$ modulo terms in $\mathscr{I}=\mathscr{M}^2_{u,v,w,\lambda}\overrightarrow{\mathscr{E}}_{u,v,w,\lambda},$ that is, the terms that are quadratic in $u,v,w,\lambda.$
\begin{equation}\label{RT7}
\begin{array}{l}
(a)\hspace{0.4cm}[p,0,0]\equiv[Au+Bv+Cw+\alpha\lambda,0,0]~~~(\mathrm{mod}\mathscr{I})\\
\\
(b)\hspace{0.4cm}[0,q,0]\equiv[0,Cu+Dv+Fw+\beta\lambda,0]~~~(\mathrm{mod}\mathscr{I})\\
\\
(c)\hspace{0.4cm}[0,0,r]\equiv[0,0,Gu+Hv+Iw+\gamma\lambda]~~~(\mathrm{mod}\mathscr{I})\\
\\
(d)\hspace{0.4cm}u[p_u,q_u,r_u]\equiv[Au,Cu,Gu]~~~(\mathrm{mod}\mathscr{I})\\
\\
(e)\hspace{0.4cm}v[p_v,q_v,r_v]\equiv[Bv,Dv,Hv]~~~(\mathrm{mod}\mathscr{I})\\
\\
(f)\hspace{0.4cm}w[p_w,q_w,r_w]\equiv[Cw,Fw,Iw]~~~(\mathrm{mod}\mathscr{I}).
\end{array}
\end{equation}
From \eqref{RT6} and \eqref{RT7} it follows that
\begin{equation}\label{RT8}
\begin{array}{l}
RT(h+t\varphi,\mathbb{Z}_2\times\mathbb{Z}_2\times\mathbb{Z}_2)=\mathscr{M}^2_{u,v,w,\lambda}
\overrightarrow{\mathscr{E}}_{u,v,w,\lambda}\oplus\mathbf{R}\left\{[Au+Bv+Cw+\alpha\lambda,0,0],\right.\\
\\
\hspace{4.4cm}[0,Cu+Dv+Fw+\beta\lambda,0],~[0,0,Gu+Hv+Iw+\gamma\lambda],\\
\\
\hspace{4.4cm}\left.[Au,Cu,Gu],~[Bv,Dv,Hv],~[Cw,Fw,Iw]\right\}.
\end{array}
\end{equation}
From \eqref{RT8} we conclude that $R(h+t\varphi,\mathbb{Z}_2\times\mathbb{Z}_2\times\mathbb{Z}_2)$ is independent of $t\varphi,$ determining \eqref{RT}. But the proof of \eqref{RT} is not complete yet. To achieve it, we have to determine \eqref{RT4}. For this purpose, we will make use of the nondegeneracy of $h.$ We know that all the generators of $\mathscr{I}$ in \eqref{RT3} are composed of quadratic or higher order terms in $u,~v,~w$ or $\lambda.$ Therefore, the following inclusion happens:
$$\mathscr{I}\subset\mathscr{M}^2_{u,v,w,\lambda}\overrightarrow{\mathscr{E}}_{u,v,w,\lambda}.$$
To prove \eqref{RT4} we must show that the inverse inclusion is also true, i.e.
$$\mathscr{M}^2_{u,v,w,\lambda}\overrightarrow{\mathscr{E}}_{u,v,w,\lambda}\subset\mathscr{I}.$$
From Nakayama's Lemma \ref{lemanaka}, the above inclusion is true provided
\begin{equation}\label{RT9}
\mathscr{M}^2_{u,v,w,\lambda}\overrightarrow{\mathscr{E}}_{u,v,w,\lambda}\subset\mathscr{I}+
\mathscr{M}^3_{u,v,w,\lambda}\overrightarrow{\mathscr{E}}_{u,v,w,\lambda}.
\end{equation}
By inspection we see that $t\varphi$ consists of terms of quadratic or higher order in $u,~v,~w$ or $\lambda.$ Hence, $t\varphi$ enters the generators of $\mathscr{I}$ in \eqref{RT3} only through cubic or higher order terms in $u,~v,~w$ or $\lambda.$ Therefore, when checking \eqref{RT9} we can assume $t\varphi\equiv0.$
For the rest of the proof we consider the thirty generators of the module $\mathscr{M}^2_{u,v,w,\lambda}\overrightarrow{\mathscr{E}}_{u,v,w,\lambda};$ they are of the form
\begin{equation}\label{RT10}
[i,0,0],~[0,i,0],~[0,0,i],~\mathrm{where}~i=\{u^2,~v^2,~w^2,~\lambda^2,~uv,~uw,~u\lambda,~vw,~v\lambda,~w\lambda\}.
\end{equation}
Moreover, we want to express these thirty generators of $\mathscr{M}^2_{u,v,w,\lambda}\overrightarrow{\mathscr{E}}_{u,v,w,\lambda}$ in terms of the twenty--seven generators of $\mathscr{I}$ in \eqref{RT3}. Since $t\varphi\equiv0,$ we can write
\begin{equation}\label{RT11}
\begin{array}{l}
(a)\hspace{1cm}p=Au+Bv+Cw+\alpha\lambda,\\
\\
(b)\hspace{1cm}q=Du+Ev+Fw+\beta\lambda,\\
\\
(c)\hspace{1cm}r=Gu+Hv+Iw+\gamma\lambda.
\end{array}
\end{equation}
This yields a $30\times27$ matrix, which we call $M$. The idea is that if we show that the rank of this matrix is $27,$ then using basic algebra we can affirm that each generator of $\mathscr{I}$ in \eqref{RT3} can be written in terms of the generators in \eqref{RT10}, hence \eqref{RT9} will follow. Now the size of the matrix $M$ makes its explicit form impossible to be written in this paper. However, based on the nondegeneracy conditions \eqref{conditions}, we will show that certain number of columns/rows can be removed, so in the end we will have showed that the rank of this matrix is $27,$ proving \eqref{RT9}. We begin by taking into account the nondegeneracy conditions $\alpha\neq0,~\beta\neq0,~\gamma\neq0,~AE\neq BD,~BF\neq CE,~AF\neq CD.$ This way we can remove from the matrix M the $12$ columns $\lambda[p,0,0],~\lambda[0,q,0],~\lambda[0,0,r],~i[p_u,q_u,r_u],~i[p_v,q_v,r_v],~i[p_w,q_w,r_w]$ where $i=uv,~uw,~vw,$ and $12$ rows $[\lambda^2,0,0],~[0,\lambda^2,0],~[0,0,\lambda^2],~[i,0,0],~[0,i,0],~[0,0,i]$ where $i=uv,~uw,~vw.$ This way we obtain a $18\times15$ matrix whose rank is precisely $15.$
\\
Next we use nondegeneracy assumptions $B|\beta|\neq|E\alpha|,$ $D|\alpha|\neq|A\beta|,$ $G|\alpha|\neq|A\gamma|,$
$C|\gamma|\neq|I\alpha|,$ $F|\gamma|\neq|I\beta|,$ $H|\beta|\neq|E\gamma|,$ to remove the $6$ rows $[w^2,0,0],$ $[0,v^2,0],$ $[0,0,u^2],$ $[\lambda w,0,0],$ $[0,\lambda v,0],$ $[0,0,\lambda u]$ and the $6$ columns $w[p,0,0],$ $v[0,q,0],$ $u[0,0,r],$ $[wr,0,0],$ $[0,vq,0],$ $[0,0,up].$ This yields a $9\times6$ matrix; we finally use the remaining nondegeneracy conditions $A\neq0,~E\neq0,~I\neq0$ (see \eqref{conditions}) to show that this matrix does have rank $6.$ Therefore, the original matrix $M$ has rank $27,$ and we have proved \eqref{RT9}.
\end{proof}

\begin{priteo}\label{teoprinc}
Let $g:\mathbf{R}^3\times\mathbf{R}\rightarrow\mathbf{R}^3$ be a bifurcation problem in three state variables commuting with the group $\mathbb{Z}_2\times\mathbb{Z}_2\times\mathbb{Z}_2$ and satisfying the nondegeneracy conditions \eqref{conditions}. Then $g$ is $\mathbb{Z}_2\times\mathbb{Z}_2\times\mathbb{Z}_2$--equivalent to
\begin{equation}\label{propnon}
\begin{array}{l}
h(x,y,z,\lambda)=(\varepsilon_1x^3+m_1xy^2+n_1xz^2+\varepsilon_2\lambda x,\varepsilon_3y^3+m_2yx^2+n_2yz^2+\varepsilon_4\lambda y,\varepsilon_5z^3+m_3zx^2+n_3zy^2+\varepsilon_6\lambda z)
\end{array}
\end{equation}
where
\begin{equation}\label{propnon1}
\begin{array}{l}
\varepsilon_1=\mathrm{sgn}(A),~\varepsilon_3=\mathrm{sgn}(E),~\varepsilon_5=\mathrm{sgn}(I),~\varepsilon_2=\mathrm{sgn}(\alpha),~\varepsilon_4=\mathrm{sgn}(\beta),~\varepsilon_6=\mathrm{sgn}(\gamma),\\
\\
\hspace{0.5cm}m_1=\displaystyle{\frac{B|\beta|}{|E\alpha|}},~m_2=\displaystyle{\frac{D|\alpha|}{|A\beta|}},
~m_3=\displaystyle{\frac{G|\alpha|}{|A\gamma|}},~
n_1=\displaystyle{\frac{C|\gamma|}{|I\alpha|}},~n_2=\displaystyle{\frac{F|\gamma|}{|I\beta|}},
~n_3=\displaystyle{\frac{H|\beta|}{|E\gamma|}}.
\end{array}
\end{equation}
Moreover,
\begin{equation}\label{nondegenr}
\begin{array}{l}
m_1\neq\varepsilon_2\varepsilon_3\varepsilon_4,~m_2\neq\varepsilon_1\varepsilon_2\varepsilon_4,
~m_3\neq\varepsilon_1\varepsilon_2\varepsilon_6,~n_1\neq\varepsilon_2\varepsilon_5\varepsilon_6,\\
n_2\neq\varepsilon_4\varepsilon_5\varepsilon_6,~n_3\neq\varepsilon_3\varepsilon_4\varepsilon_5,~
m_1m_2\neq\varepsilon_1\varepsilon_3,~m_1n_2\neq\varepsilon_3n_1.
\end{array}
\end{equation}
\end{priteo}

\begin{nota}
\begin{enumerate}
\item{} The normal form $h$ in \eqref{propnon} depends on the six parameters $m_i,n_i,i=1,2,3$ satisfying the nondegeneracy conditions \eqref{nondegenr}. These are the six modal parameters promised at the very beginning of the Subsection \ref{subsecequiv}.

\item{} The proof of Theorem \ref{teoprinc} divides into two parts. In the first part, we use the linear $\mathbb{Z}_2\times\mathbb{Z}_2\times\mathbb{Z}_2$--equivalences to transform $k$ to the normal form $h.$ In the second part, we show that the higher--order terms can be annihilated by a nonlinear $\mathbb{Z}_2\times\mathbb{Z}_2\times\mathbb{Z}_2$--equivalence. This second part actually consists entirely on the proof of Theorem \ref{teo42}.
\end{enumerate}
\end{nota}

\begin{proof}
The most general linear $\mathbb{Z}_2\times\mathbb{Z}_2\times\mathbb{Z}_2$--equivalence is given by
\begin{equation*}
Z(x,y,z,\lambda)=(ax,by,cz),~~
\Lambda(\lambda)=\sigma\lambda,~~
S(x,y,z,\lambda)=\begin{bmatrix}d&0&0\\0&e&0\\0&0&f\end{bmatrix},
\end{equation*}
where $a,~b,~c,~d,~e$ and $f$ are positive constants. Letting this equivalence act on $k(x,y,z,\lambda),$ which is given by \eqref{lowestb}, we find
\begin{equation}\label{nn}
\begin{bmatrix}d&0&0\\0&e&0\\0&0&f\end{bmatrix}k(ax,by,cz,\sigma\lambda)=
\begin{bmatrix}Ada^3x^3+Bdab^2xy^2+Cdac^2xz^2+d\alpha\lambda\sigma ax\\
                Dea^2byx^2+Eeb^3y^3+Febc^2yz^2+e\beta\lambda\sigma by\\
                Gfa^2czx^2+Hfb^2czy^2+Ifc^3z^3+f\gamma\lambda\sigma cz\end{bmatrix}.
\end{equation}
To obtain the normal form \eqref{propnon} we need
\begin{equation}\label{solnnn}
\begin{array}{l}
|A|da^3=1,~da\sigma|\alpha|=1,~|E|eb^3=1,~eb\sigma|\beta|=1,~|I|fc^3=1,~fc\sigma|\gamma|=1.
\end{array}
\end{equation}

We solve equations \eqref{solnnn} to obtain

\begin{equation}\label{solnnngf}
\begin{array}{l}
\displaystyle{d=\frac{1}{a^3|A|}},~\displaystyle{e=\frac{1}{b^3|E|}},~
\displaystyle{f=\frac{1}{c^3|I|}},~\displaystyle{\sigma=\frac{a^2|A|}{|\alpha|}},~
\displaystyle{\frac{a}{b}=\sqrt{\left|\frac{E\alpha}{A\beta}\right|}},~
\displaystyle{\frac{a}{c}=\sqrt{\left|\frac{I\alpha}{C\gamma}\right|}},~
\displaystyle{\frac{b}{c}=\sqrt{\left|\frac{I\beta}{E\gamma}\right|}}.
\end{array}
\end{equation}
Substitution of \eqref{solnnngf} into the right--hand side of \eqref{nn} yields the normal form \eqref{propnon} with $m_1,~m_2,~m_3,~n_1,~n_2$ and $n_3$ given in \eqref{propnon1}. Then we use Theorem \ref{teo42} to complete the proof.
\end{proof}

The preceding analysis of the $\mathbb{Z}_2\times\mathbb{Z}_2\times\mathbb{Z}_2$--equivariant bifurcation problem yields an example,
namely, following form of equation \eqref{propnon1}

\begin{equation}\label{propnon2}
G(x,y,z,\lambda)=
\left(
\begin{array}{l}
\varepsilon_1x^3+m_1xy^2+n_1xz^2-\lambda x\\
\\
\varepsilon_3y^3+m_2yx^2+n_2yz^2-\lambda y\\
\\
\varepsilon_5z^3+m_3zx^2+n_3zy^2-\lambda z
\end{array}
\right).
\end{equation}
We will state the next theorem whose proof is identical to the proof of Theorem 6.8 of \cite{GS79}.
\begin{priteo}\label{teominus1}
Let $H(x,y,z,\lambda)$ be a bifurcation problem with symmetry group $\Gamma=\mathbb{Z}_2\times\mathbb{Z}_2\times\mathbb{Z}_2.$ Suppose that $H$ is a small perturbation of a non-degenerate problem \eqref{propnon2}
with modal parameters $m_{1_0},m_{2_0},m_{3_0},n_{1_0},n_{2_0},n_{3_0}.$ Then $H$ is $\Gamma$--equivalent to
\begin{equation}\label{propnon3}
F(x,y,z,\lambda)=
\left(
\begin{array}{l}
\varepsilon_1x^3+m_1xy^2+n_1xz^2-\lambda x\\
\\
\varepsilon_3y^3+m_2yx^2+n_2yz^2-(\varepsilon_2+\lambda) y\\
\\
\varepsilon_5z^3+m_3zx^2+n_3zy^2-(\varepsilon_2+\varepsilon_6+\lambda )z
\end{array}
\right),
\end{equation}
where $(m_1,m_2,m_3,n_1,n_2,n_3,\lambda)$ is near $(m_{1_0},m_{2_0},m_{3_0},n_{1_0},n_{2_0},n_{3_0},0).$
\end{priteo}

The qualitative bifurcation diagrams illustrating four mode jumping possibilities are shown in Figure \ref{bif}. To explain the derivation of the Figure \ref{bif}, remark that setting \eqref{propnon3} equal to zero yields the equations
\begin{equation}\label{propnon4}
\begin{array}{l}
(a)~~x=0;~y=0;~z=0,\\
\\
(b)~~x^2=\displaystyle{\frac{\lambda}{\varepsilon_1}};~y=0;~z=0,\\
\\
(c)~~x=0;~y^2=\displaystyle{\frac{\varepsilon_2+\lambda}{\varepsilon_3}};~z=0,\\
\\
(d)~~x=0;~y=0;~z^2=\displaystyle{\frac{\varepsilon_2+\varepsilon_6+\lambda}{\varepsilon_5}},\\
\\
(e)~~\varepsilon_1x^2+m_1y^2=\lambda;~\varepsilon_3y^2+m_2x^2=\varepsilon_2+\lambda;~z=0,\\
\\
(f)~~\varepsilon_1x^2+n_1z^2=\lambda;~y=0;~\varepsilon_5z^2+m_3x^2=\varepsilon_2+\varepsilon_6+\lambda,\\
\\
(g)~~x=0;~\varepsilon_3y^2+n_2z^2=\varepsilon_2+\lambda;~\varepsilon_5z^2+n_3y^2=\varepsilon_2+\varepsilon_6+\lambda,\\
\\
(h)~~\left\{
\begin{array}{l}
\varepsilon_1x^2+m_1y^2+n_1z^2=\lambda;\\
\varepsilon_3y^2+m_2x^2+n_2z^2=\varepsilon_2+\lambda;\\
\varepsilon_5z^2+m_3x^2+n_3y^2=\varepsilon_2+\varepsilon_6+\lambda.
\end{array}
\right.
\end{array}
\end{equation}
\begin{figure}[h]
\centering
\begin{center}
\includegraphics[scale=0.2]{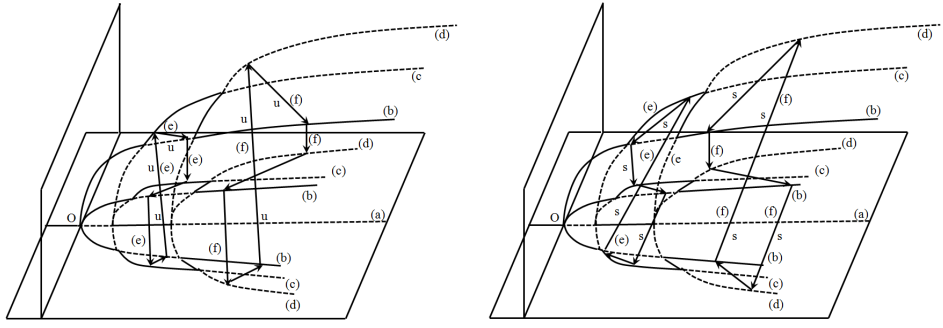}
\caption{Schematic representation of a structurally stable/unstable heteroclinic cycles derived from singularity results of the bifurcation problem with $\mathbb{Z}_2\times\mathbb{Z}_2\times\mathbb{Z}_2$ symmetry. Arrows are used to better show the jumps; they have no direction meanings. Solid lines represent stable the branches while dotted lines, the unstable ones. Parameters: $n_2,\varepsilon_2,\varepsilon_6<0,$ $n_3,\varepsilon_1,\varepsilon_3,\varepsilon_4,\varepsilon_5>0,$ $n_3<\varepsilon_3,$ $\varepsilon_3+\varepsilon_6>n_3$ and $m_2,m_3<1.$ Left figure: $\varepsilon_3\varepsilon_5+n_2n_3<0,m_1m_2<\varepsilon_1\varepsilon_3,n_1m_3<\varepsilon_1\varepsilon_5.$ Right figure: $\varepsilon_3\varepsilon_5+n_2n_3>0,m_1m_2>\varepsilon_1\varepsilon_3,n_1m_3>\varepsilon_1\varepsilon_5.$}\label{bif}
\end{center}
\end{figure}
The first seven equations in \eqref{propnon4} have real solutions; the last one can have periodic solutions. To avoid too complicated bifurcation diagrams we resume the possibility of jumping to the cases $(a)-(f).$ Computation of the conditions imposed in the parameter space in the caption of Figure \ref{bif} are easily obtained from the seven equation \eqref{propnon4} and their explicit derivation is left as an exercise to the reader. It is important to remark that when
\begin{equation*}\label{param6}
n_2,\varepsilon_2,\varepsilon_6<0,~n_3,\varepsilon_1,\varepsilon_3,\varepsilon_4,\varepsilon_5>0,~n_3<\varepsilon_3,~
\varepsilon_3+\varepsilon_6>n_3~m_2,m_3<1,
\end{equation*}
a quasi-static variation of $\lambda$ produces a smooth transition between the bifurcating branches in Figure \ref{bif}, right, and a necessity for jumping between the branches in Figure \ref{bif}, left.
\subsection{Linearized stability and $\mathbb{Z}_2\times\mathbb{Z}_2\times\mathbb{Z}_2$ symmetry}
Let $g:\mathbf{R}^3\times\mathbf{R}\rightarrow\mathbf{R}^3$ be a bifurcation problem commuting with the group $\mathbb{Z}_2\times\mathbb{Z}_2\times\mathbb{Z}_2.$ We call a solution $(x,y,z,\lambda)$ of the equation $g(x,y,z,\lambda)=0$ linearly stable if all three eigenvalues of $dg$ at $(x,y,z,\lambda)$ have positive linear part; unstable if at least one of them has negative real part. We begin by calculating the eigenvalues of $dg.$\\
Since $g$ has the form of equation \eqref{eq lema normal form} which we recall here
\begin{equation}\label{eq lema normal form1}
\begin{array}{l}
\hspace{1cm}g(x,y,z,\lambda)=(p(u,v,w,\lambda)x,~q(u,v,w,\lambda)y,~~r(u,u,w,\lambda)z)~~\mathrm{where}\\
\\
u=x^2,~~v=y^2,~~w=z^2,~~p(0,0,0,0)=0,~~~q(0,0,0,0)=0,~~~r(0,0,0,0)=0.
\end{array}
\end{equation}
The Jacobian matrix is then

\begin{equation}\label{equation dg}
dg=
\begin{array}{l}
\begin{bmatrix}
p+2up_u&2p_vxy&2p_wxz\\
2q_uxy&q+2vq_v&2q_wyz\\
2r_uxz&2r_vyz&r+2wr_w
\end{bmatrix}.
\end{array}
\end{equation}
Let $(x,y,z,\lambda)$ be a solution to $g=0.$ We find the following mode solutions:
\begin{itemize}
\vspace{0.25cm}
    \item [(a)] Trivial solution: when $x=y=z=0;$
    \item [(b)] $x$--mode solution: $p(x^2,0,0,\lambda)=0,~y=z=0,x\neq0;$
    \item [(c)] $y$--mode solution: $q(0,y^2,0,\lambda)=0,~x=z=0,y\neq0;$
    \item [(d)] $z$--mode solution: $r(0,0,z^2,\lambda)=0,~x=y=0,z\neq0;$
    \item [(e)] $xy$--mode solution: $p(x^2,y^2,0,\lambda)=q(x^2,y^2,0,\lambda)=0,~z=0,~x\neq0,~y\neq0;$
    \item [(f)] $xz$--mode solution: $p(x^2,0,z^2,\lambda)=r(x^2,0,z^2,\lambda)=0,~y=0,~x\neq0,~z\neq0;$
    \item [(g)] $yz$--mode solution: $q(0,y^2,z^2,\lambda)=r(0,y^2,z^2,\lambda)=0,~x=0,~y\neq0,~z\neq0;$
    \item [(h)] $xyz$--mode solution: $p(x^2,y^2,z^2,\lambda)=q(x^2,y^2,z^2,\lambda)=r(x^2,y^2,z^2,\lambda)=0,~x\neq0,~y\neq0,~z\neq0;$
\end{itemize}
To analyze the stability of these solutions we need the explicit form of the eigenvalues of the Jacobian matrix \eqref{equation dg}. We have
    \begin{itemize}
    \vspace{0.25cm}
    \item [(a)] Trivial solution: when $x=y=z=0$ with eigenvalues: $p,q,r.$
    \item [(b)] $x$--mode solution: $p(x^2,0,0,\lambda)=0,~y=z=0,x\neq0$ with eigenvalues: $2up_u,q,r.$
    \item [(c)] $y$--mode solution: $q(0,y^2,0,\lambda)=0,~x=z=0,y\neq0$ with eigenvalues: $p,2vq_v,r.$
\item [(d)] $z$--mode solution: $r(0,0,z^2,\lambda)=0,~x=y=0,z\neq0$ with eigenvalues: $p,$ $q,$ $2wr_w.$
\item [(e)] $xy$--mode solution: $p(x^2,y^2,0,\lambda)=q(x^2,y^2,0,\lambda)=0,~z=0,~x\neq0,~y\neq0$ with eigenvalues:
$r,vq_v+up_u\pm\sqrt{u^2p_u^2-2up_uvq_v+v^2q_v^2+4p_vx^2y^2q_u}.$
    \item [(f)] $xz$--mode solution: $p(x^2,0,z^2,\lambda)=r(x^2,0,z^2,\lambda)=0,~y=0,~x\neq0,~z\neq0$ with eigenvalues:
$q,vq_v+up_u\pm\sqrt{w^2r_w^2-2up_uwr_w+u^2p_u^2+4r_ux^2z^2p_w}.$
    \item [(g)] $yz$--mode solution: $q(0,y^2,z^2,\lambda)=r(0,y^2,z^2,\lambda)=0,~x=0,~y\neq0,~z\neq0$ with eigenvalues:
$p,~vq_v+wr_w\pm\sqrt{v^2q_v^2-2vq_vwr_w+w^2r_w^2+4q_wy^2z^2r_v}.$
\item [(h)] $xyz$--mode solution: $p(x^2,y^2,z^2,\lambda)=q(x^2,y^2,z^2,\lambda)=r(x^2,y^2,z^2,\lambda)=0,~x\neq0,~y\neq0,~z\neq0.$
A short calculation with Matlab, for example, allows finding the explicit form of the eigenvalues $\mu_1,\mu_2$ and $\mu_3$, which are too large to be exposed here. In all these cases $(a)-(h)$ the stability of the solutions is given by the sign of the linear part, as indicated above.
\end{itemize}

\section{The Cayley graph of the $\mathbb{Z}_2\times\mathbb{Z}_2\times\mathbb{Z}_2$ group}\label{sec4}
In this section we build an oscillatory network with the symmetry of the $\mathbb{Z}_2\times\mathbb{Z}_2\times\mathbb{Z}_2$ group, and describe the elements of this group, as the relationships between them. We first to represent the group by a Cayley diagram. As shown in \cite{murza}, a Cayley diagram is a set of nodes and arrows/edges, connected in such a way to a represent a group. The nodes represent the group elements while the arrows are used to describe how the generators act on the group elements. For more details on how to construct a $\Gamma$--equivariant ODE via the Cayley graph of the $\Gamma$ group, see \cite{ADSW}.

\begin{figure}[ht]
\centering
\begin{center}
\includegraphics[scale=0.35]{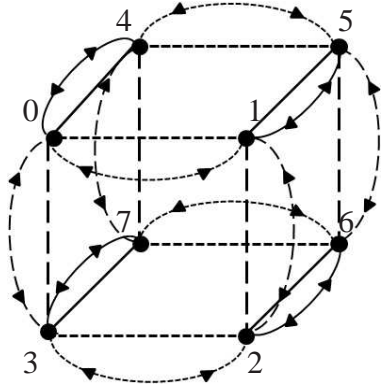}
\caption{A Cayley graph of the $\mathbb{Z}_2\times\mathbb{Z}_2\times\mathbb{Z}_2$ group. We use solid arrows to represent the left-multiplication with $a,$ dashed arrows left multiplication with $b,$ while dotted lines represent left multiplication with $c,$ the three generators of this group.}\label{second_figure}
\begin{picture}(0,0)\vspace{0.5cm}
\put(-64,149){\large{$0$}}
\put(20,149){\large{$1$}}
\put(20,54){\large{$2$}}
\put(-64,54){\large{$3$}}
\put(-29,183){\large{$4$}}
\put(54,183){\large{$5$}}
\put(54,107){\large{$6$}}
\put(-29,107){\large{$7$}}
\end{picture}\vspace{-0.3cm}
\end{center}
\end{figure}

The Cayley graph for $\mathbb{Z}_2\times\mathbb{Z}_2\times\mathbb{Z}_2$ is shown in Figure \eqref{second_figure}. By following the works of Stork et al. \cite{Stork} and Ashwin et al. \cite{ADSW} we build an ODE system which has the symmetry of  $\mathbb{Z}_2\times\mathbb{Z}_2\times\mathbb{Z}_2$ group. The action of the group $\mathbb{Z}_2\times\mathbb{Z}_2\times\mathbb{Z}_2$ on the nodes can be written by taking into account the relationship between their generators:
\begin{equation}\label{elements}
a=(0~3)(1~2)(4~7)(5~6),~~b=(0~4)(1~5)(3~7)(2~6),~~c=(3~2)(7~6)(0~1)(4~5),
\end{equation}
with the relationship between them
\begin{equation}\label{commutators}
\begin{array}{l}
\hspace{1cm}a^2=b^2=c^2=\mathbbm{1},~~~ab=ba,~~~ac=ca\\
bc=cb,~~~ab\cdot ac=bc,~~~ab\cdot bc=ac,~~~ac\cdot bc=ab.
\end{array}
\end{equation}
Nodes related by $a$ have assigned coupling $g$ nodes related by $b$ have assigned coupling $h$; then, from the permutations in \eqref{elements}, we construct the system \eqref{array 4 eq} with the symmetry of the $\mathbb{Z}_2\times\mathbb{Z}_2\times\mathbb{Z}_2$ group.
\begin{equation}\label{array 4 eq}
\begin{array}{l}
\dot{x}_0=f(x_0)+g(x_{0},x_4)+h(x_1,x_5)\\
\dot{x}_1=f(x_1)+g(x_3,x_7)+h(x_2,x_{6})\\
\dot{x}_2=f(x_2)+g(x_0,x_3)+h(x_{1},x_2)\\
\dot{x}_3=f(x_3)+g(x_4,x_7)+h(x_5,x_6)\\
\dot{x}_4=f(x_4)+g(x_3,x_2)+h(x_7,x_6)\\
\dot{x}_5=f(x_5)+g(x_0,x_{1})+h(x_4,x_5)\\
\dot{x}_6=f(x_6)+g(x_{0},x_{5})+h(x_1,x_{4})\\
\dot{x}_7=f(x_7)+g(x_{2},x_{7})+h(x_{3},x_6)\\
\end{array}
\end{equation}
where $f:\mathbf{R}\rightarrow\mathbf{R}$ and $g,~h:\mathbf{R}^2\rightarrow\mathbf{R}$.

\section{Weak Coupling}\label{sec5}

As shown in \cite{ADSW} majority of observable physical systems ar the result of the interaction between nearly identical subsystems, and these can be described by the perturbation of independent systems. It has been shown in Ashwin and Swift \cite{Ashwin_Swift}, that even oscillatory systems composed by strongly coupled subsystems, can be described within the "weak coupling limit". Based on these observations, we are allowed to assume that our system behave as a weakly coupled subsystems. Another assumption that we make is that our system is formed by oscillators that are dissipative subsystems with attracting and unique periodic orbits. We will devote especial interest to the dynamics of the phases of the weakly coupled oscillators. This situation becomes evident if taking into account the fact that in absence of coupling there is an attracting $n$--torus with a particular angle corresponding to every oscillator. A problem may arise if one focuses only on the phase differences; in this case we can switch to suitable coordinates that convert the dynamics into a linear flow and in addition this has no effect on the phase differences. This approach is related to the Hopf bifurcation, but instead of examining small amplitude oscillations near a Hopf bifurcation point, we make a weak coupling approximation. In the theory developed in \cite{Ashwin_Swift} dynamics takes place on a torus that is normally hyperbolic, so with small coupling, the torus persists and there is a slow evolution of the phase differences. Moreover, it has been pointed out that while the Hopf bifurcation theory gives local information, in the weak coupling case the results hold globally on the n-torus. In this sense, our approach based on weak coupling differs essentially from that carried out in \cite{Melbourne}. We do not use the Birkhoff normal form truncated at the cubic order as in \cite{Melbourne}; in fact our approach built on the most general $\mathbb{Z}_2\times\mathbb{Z}_2\times\mathbb{Z}_2$--equivariant system via Cayley graph, will mainly use group theoretical ideas to prove the existence of heteroclinic cycles in such a system.

Based on the facts explained above, in the weak coupling case equation \eqref{array 4 eq} reads
\begin{equation}\label{generic weak coupling equation}
\dot{x}_i=f(x_i)+\epsilon g_i(x_0,\ldots,x_{7})
\end{equation}
for $i=0,\ldots,7.$ System \eqref{generic weak coupling equation} commutes with $\mathbb{Z}_2\times\mathbb{Z}_2\times\mathbb{Z}_2$ on $\mathbb{T}^{8},$ both $f$ and $g_i$ being of the class $C^{\infty}.$ We denote by $\epsilon$ the coupling, and it has low values. As in \cite{ADSW}, \cite{Stork} or \cite{Ashwin_Swift} we assume the existence of a hyperbolic limit cycle for $\dot{x}$; in addition, it is assumed to be stable. This is so because of the following. Under no coupling, there is an attracting $8$--torus $\mathbb{T}^8$, with one angle/oscillator. Next we chose coordinates in such a way that the dynamics is a linear flow in the $(1,1,1,1,1,1,1,1)$ direction; this yields no change in the phase differences. Since the torus is hyperbolic, at small $\epsilon$, it persists.\\

Weakly coupled oscillators afford another reason for not having in mind only irreducible representations of $\mathbb{Z}_2\times\mathbb{Z}_2\times\mathbb{Z}_2,$ see \cite{Ashwin_Swift}. In our case of a network with $\mathbb{Z}_2\times\mathbb{Z}_2\times\mathbb{Z}_2$ symmetry, we deal with $8$ limit cycles that are stable and hyperbolic for $\epsilon=0.$ This leads to the conclusion that in the asymptotic limit, the dynamics factors into the dynamical behavior of eight limit cycles.
As explained above, we can assume that every one of these limit cycles are hyperbolic for small $\epsilon$ and this justifies expressing the dynamics of the system as an ODE in terms of its phases, i.e. an ODE on $\mathbb{T}^{8}$  which is $\mathbb{Z}_2\times\mathbb{Z}_2\times\mathbb{Z}_2$--equivariant.
\small
\begin{table}
\centering
\begin{center}
\caption{Isotropy subgroups and fixed point subspaces for the $\mathbb{Z}_2\times\mathbb{Z}_2\times\mathbb{Z}_2\times\mathbb{S}^1$ action on $\mathbb{T}^{8}$.}\label{table grande}
% on $\mathbf{V}.$}
\end{center}
\begin{tabular}{ccccc}
\toprule
%\multicolumn{4}{c}{Item} \\
%\cmidrule(r){1-4}
$\Sigma$ & $\mathrm{Fix}(\Sigma)$ & Generators & $\mathrm{dim~Fix}(\Sigma)$\\
\midrule
\centering
$\mathbb{Z}_2\times\mathbb{Z}_2\times\mathbb{Z}_2$ & $(0,0,0,0,0,0,0,0)$ & $(a,0),(b,0),(c,0)$ &$0$\\
$(\mathbb{Z}_2\times\mathbb{Z}_2\times\mathbb{Z}_2)^{a}$ & $(0,0,0,0,\pi,\pi,\pi,\pi)$ & $(a,\pi),(\mathbbm{1},\pi),(\mathbbm{1},\pi)$ &$0$\\
$(\mathbb{Z}_2\times\mathbb{Z}_2\times\mathbb{Z}_2)^{b}$ & $(0,\pi,0,\pi,0,\pi,0,\pi)$ & $(\mathbbm{1},\pi),(b,\pi),(\mathbbm{1},\pi)$ &$0$\\
$(\mathbb{Z}_2\times\mathbb{Z}_2\times\mathbb{Z}_2)^{c}$ & $(0,0,\pi,\pi,0,0,\pi,\pi)$ & $(\mathbbm{1},\pi),(\mathbbm{1},\pi),(c,\pi)$ &$0$\\
$(\mathbb{Z}_2\times\mathbb{Z}_2)^a$ & $(\pi,\pi,\phi,\phi,\mathbbm{1},\mathbbm{1},\mathbbm{1},\mathbbm{1})$ & $(a,\pi),(\mathbbm{1})$ &$1$\\
$(\mathbb{Z}_2\times\mathbb{Z}_2)^b$ & $(\pi,\phi,\pi,\phi,\mathbbm{1},\mathbbm{1},\mathbbm{1},\mathbbm{1})$ & $(b,\pi),(\mathbbm{1})$ &$1$\\
$(\mathbb{Z}_2\times\mathbb{Z}_2)^{c}$ & $(\phi,\phi,\pi,\pi,\mathbbm{1},\mathbbm{1},\mathbbm{1},\mathbbm{1})$ & $(ab,\pi),(\mathbbm{1})$ &$1$\\
$\mathbb{Z}_2$&$(0,\phi_1,0,\phi_1,\phi_2,\phi_3,\phi_2,\phi_3)$& $(a,0)$ &$3$\\
$\tilde{\mathbb{Z}}_2$&$(0,\phi_1,\pi,\phi_1+\pi,\phi_2,\phi_3,\phi_2+\pi,\phi_3+\pi)$& $(a,\pi)$ &$3$\\
\bottomrule
\end{tabular}
\end{table}
\normalsize

\begin{priteo}
Consider the action of $\mathbb{Z}_2\times\mathbb{Z}_2\times\mathbb{Z}_2\times\mathbb{S}^1$ on $\mathbb{T}^8.$ The isotropy subgroups, their generators and the corresponding dimension of their fixed-point subspaces are those listed in Table \eqref{table grande}.
\end{priteo}
\begin{proof}
We calculate one example, for the zero and one-dimensional fixed-point subspaces, the other cases being left as an exercise for the reader.
$(a)$ Consider the action of $\mathbb{Z}_2\times\mathbb{Z}_2\times\mathbb{Z}_2$ on $\mathbb{T}^8.$ We have
\begin{equation*}\label{teorema tabla example}
\begin{array}{l}
\begin{bmatrix}
\cos\phi_1&-\sin\phi_1&0&0&0&0&0&0\\
\sin\phi_1&\cos\phi_1&0&0&0&0&0&0\\
0&0&\cos\phi_2&-\sin\phi_2&0&0&0&0\\
0&0&\sin\phi_2&\cos\phi_2&0&0&0&0\\
0&0&0&0&\cos\phi_3&-\sin\phi_3&0&0\\
0&0&0&0&\sin\phi_3&\cos\phi_3&0&0\\
0&0&0&0&0&0&\mathbbm{1}&0\\
0&0&0&0&0&0&0&\mathbbm{1}
\end{bmatrix}
\begin{bmatrix}
a\\b\\c\\d\\e\\f\\g\\h
\end{bmatrix}=
\begin{bmatrix}
a\cos\phi_1-b\sin\phi_1\\
b\cos\phi_1+a\sin\phi_1\\
c\cos\phi_2-d\sin\phi_2\\
d\cos\phi_2+c\sin\phi_2\\
e\cos\phi_1-f\sin\phi_1\\
f\cos\phi_1+e\sin\phi_1\\
g\\
h
\end{bmatrix}=
\begin{bmatrix}
a\cos\phi\\
b\cos\phi\\
c\cos\phi\\
d\cos\phi\\
e\cos\phi\\
f\cos\phi\\
g\\
h
\end{bmatrix}=
\begin{bmatrix}
\pm a\\
\pm b\\
\pm c\\
\pm d\\
\pm e\\
\pm f\\
g\\
h
\end{bmatrix}.
\end{array}
\end{equation*}
because $\phi_1,\phi_2,\phi_3=\{0,\pi\}.$ There are just two possible values, i.e. $0$ and $\pi$ for any arbitrary point on $\mathbb{T}^8$ that allow to be fixed by the group. The three choices for the first three lines are deduced from the action of the elements of $\mathbb{Z}_2\times\mathbb{Z}_2\times\mathbb{Z}_2.$\\
\end{proof}

Ashwin and Swift showed in \cite{Ashwin_Swift} that for small $\epsilon$ one can average the equations forming the system. This can be interpreted as involving a phase--shift symmetry in the dynamics; it acts on $\mathbb{T}^8$ by translating the phases along the diagonal;
$$R_{\theta}(\phi_0,\ldots,\phi_{7}):=(\phi_0+\theta,\ldots,\phi_{7}+\theta),$$
for $\theta\in\mathbb{S}^1.$

In Table \eqref{table grande} we classify the isotropy subgroups and fixed point subspaces for the $\mathbb{Z}_2\times\mathbb{Z}_2\times\mathbb{Z}_2\times\mathbb{S}^1$ action on $\mathbb{T}^{8}.$\\
Since now on, our interest focuses in the three-dimensional space $\mathbb{Z}_2;$ it contains several one- and zero-dimensional fixed-point subspaces. These are fixed by the elements in the normalizer of $\mathrm{Fix}(\mathbb{Z}_2).$

\subsection{Dynamics of the $\theta_1,\theta_2$ and $\theta_3$ angles in $\mathrm{Fix}(\mathbb{Z}_2)$}\label{onetorus_theor}
We can define coordinates in $\mathrm{Fix}(\mathbb{Z}_2) $ by taking a basis
\begin{equation}\label{basis}
\begin{array}{l}
e_1=-\displaystyle{\frac{1}{4}}(1,1,1,1,-1,-1,-1,-1),\\
\\
e_2=-\displaystyle{\frac{1}{4}}(1, -1, 1, -1, 1, -1, 1, -1),\\
\\
e_3=-\displaystyle{\frac{1}{4}}(1, 1, -1, -1, 1, 1, -1, -1),
\end{array}
\end{equation}
and consider the space spanned by $\{e_1,e_2,e_3\}$ parametrized by $\{\theta_1,\theta_2,\theta_3\}:$ $\sum_{n=1}^3\theta_ne_n.$
By using these coordinates, we construct the following family of three-dimensional differential systems which satisfies the symmetry of $\mathrm{Fix}(\mathbb{Z}_2)$.
\begin{equation}\label{systema ejemplo}
\left\{
\begin{array}{l}
\dot{\theta_1}=u\sin{\theta_1}\cos{\theta_2}+\epsilon\sin{2\theta_1}\cos{2\theta_2}\\
\\
\dot{\theta_2}=u\sin{\theta_2}\cos{\theta_3}+\epsilon\sin{2\theta_2}\cos{2\theta_3}\\
\\
\dot{\theta_2}=u\sin{\theta_3}\cos{\theta_1}+\epsilon\sin{2\theta_3}\cos{2\theta_1}+q(1-\cos(\theta_1-\theta_2))\sin2\theta_3,\\
\end{array}
\right.
\end{equation}
where $u,\epsilon$ and $q$ are real constants.

We will show that this vector field contains a heteroclinic cycle which may be asymptotically stable, essentially asymptotically stable or completely unstable, depending on the values of $u,\epsilon$ and $q.$ As in \cite{ADSW}, we can assume, without loss of genericity that the space $\mathrm{Fix}(\mathbb{Z}_2)$ is normally attracting for the dynamics; therefore, we assume that the dynamics within the fixed-point space determines the stability of the full system.
In the following we will show that the planes $\phi_1=0~(\mathrm{mod}~\pi),~\phi_2=0~(\mathrm{mod}~\pi)$ are invariant under the flow of \eqref{systema ejemplo}.

Let $\mathcal{X}$ be the vector field of system \eqref{systema ejemplo}.
\begin{defi}
We call a trigonometric invariant algebraic surface $h(\theta_1,\theta_2,\theta_3)=0,$ if it is invariant by the flow of \eqref{systema ejemplo}, i.e. there exists a function $K(\theta_1,\theta_2, \theta_3)$ such that
\begin{equation}\label{campo}
\mathcal{X}h=\frac{\partial h}{\partial\theta_1}\dot{\theta_1}+\frac{\partial h}{\partial\theta_2}\dot{\theta_2}
+\frac{\partial h}{\partial\theta_3}\dot{\theta_3}=Kh.
\end{equation}
\end{defi}

\begin{lema}
Functions $\sin\theta_1,~\sin\theta_2$ and $\sin\theta_3$ are trigonometric invariant algebraic surfaces for system \eqref{systema ejemplo}.
\end{lema}
\begin{proof}
We can write the system \eqref{systema ejemplo} in the form
\begin{equation}\label{systema ejemplo1}
\left\{
\begin{array}{l}
\dot{\theta_1}=\sin{\theta_1}\left(u\cos{\theta_2}+2\epsilon\cos{\theta_1}\cos{2\theta_2}\right)\\
\\
\dot{\theta_2}=\sin{\theta_2}\left(u\cos{\theta_3}+2\epsilon\cos{\theta_2}\cos{2\theta_3}\right)\\
\\
\dot{\theta_3}=\sin{\theta_3}\left(u\cos{\theta_1}+2\epsilon\cos{2\theta_1}\cos{\theta_3}+2q(1-\cos(\theta_1-\theta_2))\cos\theta_3\right)\\
\end{array}
\right.
\end{equation}
If we choose $h_1=\sin\theta_1,$ then $\mathcal{X}h_1=\cos{\theta_1}\sin{\theta_1}\left(u\cos{\theta_2}+2\epsilon\cos{\theta_1}\cos{2\theta_2}\right)$
so $K_1=\cos{\theta_1}\left(u\cos{\theta_2}+2\epsilon\cos{\theta_1}\cos{2\theta_2}\right).$ The remaining cases follow similarly.
\end{proof}

Since the planes $\theta_i=0(\mathrm{mod}~\pi)$ are invariant under the flow of \eqref{systema ejemplo}, it is clear that $(0,0,0),~(\pi,0,0),~(0,\pi,0)$, and $(0,0,\pi)$ are equilibria for \eqref{systema ejemplo}. Next we search the existence of heteroclinic cycles in system \eqref{systema ejemplo}. The first step consists in linearizing it about the equilibria (i.e. the zero-dimensional fixed points). The idea is proving that there are three-dimensional fixed-point spaces $\mathrm{Fix}(\mathbb{Z}_2)$ and $\mathrm{Fix}(\tilde{\mathbb{Z}}_2)$ which connect these fixed points; this would allow the existence of such a heteroclinic network between the equilibria.\\
Let's assume
\begin{equation}\label{param}
\begin{array}{l}
\displaystyle{|\epsilon|<\frac{u}{2}}~~ \mathrm{and}~~
\displaystyle{|\epsilon+2q|<\frac{u}{2}}.
\end{array}
\end{equation}

\begin{table}
\centering
\begin{center}
\caption{Eigenvalues of the flow of equation \eqref{systema ejemplo}, at the four non-conjugate zero-dimensional fixed points.}\label{table fixed points}
\end{center}
\begin{tabular}{c|c|c|c|c}
\toprule
$\mathrm{Fix}(\Sigma)$ & $(\theta_1,\theta_2,\theta_3)$ & $\lambda_1$ & $\lambda_2$& $\lambda_3$\\
\midrule
\centering
$\mathbb{Z}_2\times\mathbb{Z}_2\times\mathbb{Z}_2$& $(0,0,0)$ & $u+2\epsilon$ &$u+2\epsilon$&$u+2\epsilon$\\
\midrule
$(\mathbb{Z}_2\times\mathbb{Z}_2\times\mathbb{Z}_2)^a$ & $(\pi,0,0)$ & $-u+2\epsilon$ &$u+2\epsilon$&$-u+2\epsilon+4q$\\
$(\mathbb{Z}_2\times\mathbb{Z}_2\times\mathbb{Z}_2)^b$ & $(0,\pi,0)$ & $-u+2\epsilon$ &$-u+2\epsilon$&$u+2\epsilon+4q$\\
$(\mathbb{Z}_2\times\mathbb{Z}_2\times\mathbb{Z}_2)^c$ & $(0,0,\pi)$ & $u+2\epsilon$ &$-u+2\epsilon$&$-u+2\epsilon$\\
\bottomrule
\end{tabular}
\end{table}

We use the criteria of Krupa and Melbourne \cite{Krupa} to study the stability of the heteroclinic cycle. We have
\begin{priteo}
There exists the possibility of a heteroclinic cycle in the following way:
\begin{equation}\label{flechas}
\begin{array}{l}
\cdots\xrightarrow{(\mathbb{Z}_2\times\mathbb{Z}_2)^a}(\mathbb{Z}_2\times\mathbb{Z}_2\times\mathbb{Z}_2)^a\xrightarrow{(\mathbb{Z}_2\times\mathbb{Z}_2)^b}
(\mathbb{Z}_2\times\mathbb{Z}_2\times\mathbb{Z}_2)^b\xrightarrow{(\mathbb{Z}_2\times\mathbb{Z}_2)^c}
(\mathbb{Z}_2\times\mathbb{Z}_2\times\mathbb{Z}_2)^c
\xrightarrow{(\mathbb{Z}_2\times\mathbb{Z}_2)^a}\cdots
\end{array}
\end{equation}
The stability of the heteroclinic cycle is:
\begin{itemize}
\vspace{0.25cm}
\item [(a)] asymptotically stable if
\begin{equation}\label{eq1 teorema estabilidad}
\begin{array}{l}
\displaystyle{u>0~\mathrm{and}~q<\frac{3u}{4}-\frac{\epsilon}{2}},
\end{array}
\end{equation}
\item [(b)] unstable but essentially asymptotically stable if
\begin{equation}\label{eq1 teorema estabilidad repetition}
\begin{array}{l}
\displaystyle{u>0~\mathrm{and}~\frac{3u}{4}-\frac{\epsilon}{2}<q<\frac{u}{2}-\frac{(u+2\epsilon)^3}{(-u+2\epsilon)^2}}.
\end{array}
\end{equation}
\item [(c)] completely unstable if $u<0.$
\end{itemize}
\end{priteo}
\begin{proof}

The stability of the cycle is expressed by
\begin{equation}\label{stability krupa1}
\rho=\prod_{i=1}^3\rho_i,~~\mathrm{where}~~\rho_i=\mathrm{min}\{c_i/e_i,1-t_i/e_i\}.
\end{equation}
Here $c_i$ stands for the contracting eigenvalues, $e_i$ the expanding eigenvalues while $t_i$ the transversal ones.
For the heteroclinic cycle we have
\begin{equation}\label{rho values11}
\rho_1=
\left\{
\begin{array}{l}
\displaystyle{\frac{2u-4q}{u+2\epsilon}}~\mathrm{if}~q<\frac{3u}{4}-\frac{\epsilon}{2}\\
\\
\displaystyle{\frac{-u+2\epsilon}{u+2\epsilon}}~\mathrm{if}~q>\frac{3u}{4}-\frac{\epsilon}{2}
\end{array}
\right.\hspace{0.25cm}
\displaystyle{\rho_2=\frac{-u+2\epsilon}{u+2\epsilon+4q}},~\rho_3=\displaystyle{\frac{-u+2\epsilon}{u+2\epsilon}},
\end{equation}
so from equations \eqref{rho values11} we obtain
\begin{equation}\label{rho values13}
\rho=
\left\{
\begin{array}{l}
\displaystyle{\frac{(-u+2\epsilon)^2(2u-4q)}{(u+2\epsilon)^2(u+2\epsilon+4q)}}~\mathrm{if}~u>0~\mathrm{and}~q<\frac{3u}{4}-\frac{\epsilon}{2},\\
\\
\displaystyle{\frac{(-u+2\epsilon)^3}{(u+2\epsilon)^2(u+2\epsilon+4q)}}~\mathrm{if}~u>0~\mathrm{and}~q>\frac{3u}{4}-\frac{\epsilon}{2}.
\end{array}
\right.
\end{equation}
Then the proof follows by applying Theorem $2.4$ in \cite{Krupa}.
\end{proof}

For any value $u>0$ we get $\displaystyle{\frac{3u}{4}-\frac{\epsilon}{2}<q<\frac{u}{2}-\frac{(u+2\epsilon)^3}{(-u+2\epsilon)^2}};$ hence, there are values of $q$ that allow existence of essentially asymptotic stable heteroclinic connections. In addition to the information in \cite{Melbourne}, our results indicate that there exists an attracting heteroclinic cycle even when the linearly analized stability of $\mathrm{Fix}(\mathbb{Z}_2\times\mathbb{Z}_2)^a$ yields an expanding transverse eigenvalue.

\paragraph{\bf Acknowledgements}
Adrian Murza was supported by a grant of Romanian National Authority for Scientific Research and Innovation, CNCS-UEFISCDI, project number PN-II-RU-TE-2014-4-0657.

\end{document}